\documentclass[11pt,leqno]{article}

\usepackage[paperwidth=8.5in, textheight=9.5in]{geometry}
\usepackage{amsmath,amsthm,amssymb}
\usepackage{tikz}
\usetikzlibrary{arrows,decorations.pathmorphing,backgrounds,positioning,fit,petri}
\usepackage{enumerate,array,float,verbatim}
\usepackage[labelsep=space]{caption}
\usepackage{mathrsfs}

\theoremstyle{plain}
\newtheorem{theorem}{Theorem}[section]
\newtheorem{lemma}[theorem]{Lemma}

\theoremstyle{definition}

\numberwithin{subcase}{case}
\numberwithin{subsubcase}{subcase}
\numberwithin{table}{section}
\numberwithin{equation}{section}

\newcommand{\singleedge}[5]{\begin{tikzpicture}%
[node distance=0.8cm,baseline={(left.base)},%
pre/.style={<-,shorten <=1pt,>=angle 45},%
post/.style={->,shorten >=1pt,>=angle 45}],%
rectangle/.style={inner sep=0pt,minimum size=4mm}];%
\node[rectangle] (left) {$\strut$#1};%
\node[rectangle] (right) [right=of left] {$\strut$#2}%
edge [#4,#5] node[yshift=5pt] {#3} (left);%
\end{tikzpicture}}
\newcommand{\solidedge}[3]{\singleedge{#1}{#2}{#3}{pre}{solid}}
\newcommand{\dashededge}[3]{\singleedge{#1}{#2}{#3}{pre}{dashed}}
\newcommand{\reversedsolidedge}[3]{\singleedge{#1}{#2}{#3}{post}{solid}}
\newcommand{\reverseddashededge}[3]{\singleedge{#1}{#2}{#3}{post}{dashed}}

\newcommand{\Gammaarrow}{\Gamma_{\!\to}}

\newcommand{\infinity}{\infty}
\newcommand{\set}[1]{\left\{#1\right\}}
\newcommand{\setof}[2]{\left\{{#1}\mid{#2}\right\}}

\newcommand{\abs}[1]{\left\vert{#1}\right\vert}
\newcommand{\rationals}{\mathbb Q}%
\newcommand{\complexes}{\mathbb C}

\newcommand{\closedray}[1]{{\left[{#1},{\infinity}\right)}}

\newcommand{\charprod}[3]{\left<{#1},\strut{#2}\right>_{#3}}

\DeclareMathOperator{\ind}{ind}
\DeclareMathOperator{\sgn}{sgn}

\DeclareMathOperator{\vertices}{{\mathscr V}}
\DeclareMathOperator{\edges}{{\mathscr E}}
\DeclareMathOperator{\In}{In}

\begin{document}
\title{An extension of a theorem and errata for 
``A Class of Representations of Hecke Algebras''}
\author{Dean Alvis}

\date{}

\maketitle

\begin{abstract}
By Theorem~1.12 of the paper ``A Class of Representations of Hecke Algebras'',
if $W$ is a Coxeter group
whose proper parabolic 
subgroups are finite, and if the module of
a finite $W$-digraph $\Gamma$ is isomorphic to the module
of a $W$-graph, then $\Gamma$ must be acyclic.  
Here we extend this result to 
 Coxeter groups with finite dihedral 
 parabolic subgroups and 
$W$-graphs with arbitrary scalar edge labels. 
Also, errata for the paper 
are listed in the last section. 
\end{abstract}


\section{An extension of Theorem~1.12 of  \cite{digraphpaper}}

Let $(W,S)$ be a Coxeter system with presentation
\[
W =
\left<
s \in S
\, \mid \,
(rs)^{n(r,s)}=e \text{ for $r,s\in S$ whenever $n(r,s)<\infinity$}
\right> ,
\]
where $n(s,s) = 1$ and $1<n(r,s)=n(s,r)\le\infinity$ for $r,s\in S$, $r \ne s$.
Let $\ell$ be the length function of $(W,S)$.
Let $u$ be an indeterminate over $\complexes$, 
and let $H$ be the Hecke algebra
of $(W,S)$ over $\rationals(u)$.  

See
\cite{digraphpaper} for the definition of the notion of {\it $W$-digraph}.
For $\Gamma$ a $W$-digraph and $\beta \in \vertices(\Gamma)$, 
define $\In(\beta)$ to be the set of all $s \in S$ such that
$\Gamma$ has an edge of the form
\solidedge{$\alpha$}{$\beta$}{$s$}
or
\dashededge{$\alpha$}{$\beta$}{$s$} 
for some $\alpha \in \vertices(\Gamma)$. 
Observe $\beta$ is a source (sink) in $\Gamma$ if and only if 
$\In(\beta) = \emptyset$ ($\In(\beta) = S$, respectively). 
For $J \subseteq S$, put
\[
N_{\Gamma}(J) =
\abs{ \setof{\beta \in \vertices(\Gamma)} {\In(\beta) = J} }.
\]
(This definition will only be applied when $\Gamma$ is finite, i.e. 
when $\vertices(\Gamma)$ and $\edges(\Gamma)$ are finite.)

Now, let $\Psi$ be a $W$-graph over the subfield $F$ of $\complexes$, 
in the sense of \cite{gyojawgraph}, Definition~2.1, with 
vertex-labeling function $x \mapsto I_x \subseteq S$, $x \in X$, 
edge-labeling function $\mu : X \times X \rightarrow F$, and with 
the indeterminate $u$ here playing the role of
 $q^{1/2}$ in \cite{gyojawgraph}.
For  $J \subseteq S$, put
\[
N_{\Psi}(J) =
\abs{ \setof{x \in \vertices(\Psi)} { I_x = J} }.
\]


The goal of this section is to prove the following.

\begin{theorem}
If $n(s,t) < \infinity$ for $s, t \in S$, 
$\Gamma$ is a finite $W$-digraph, $\Psi$ is a $W$-graph
over a subfield $F$ of $\complexes$, 
and $M(\Gamma)^F = F(u) \otimes_{\rationals(u)} M(\Gamma)$ is 
isomorphic to $M(\Psi)$ as $H^F$-modules, then
the following hold:
\begin{enumerate}[{\upshape(i)}]
\item
$N_{\Gamma}(J) = N_{\Psi}(J)$ for all $J \subseteq S$.
\item
$\Gamma$ is acyclic.
\end{enumerate}
\label{theorem:wgraphacyclic}
\end{theorem}


We require the following results.
For $F$ a subfield of $\complexes$, $\lambda$ a linear character
of $H^F$, and $V$ an $H^F$-module, define 
\[
V_\lambda =
\setof{v \in V}{h v = \lambda(h) v \text{ for all } h \in H^F}
\]
and 
\[
\charprod{V}{\lambda}{H^F}
=
\dim_{F(u)} V_{\lambda}.
\]
Also, let $\ind^F$ and $\sgn^F$ be the linear characters of $H^F$ determined
by $\ind^F(T_w) = u_w = u^{2 \ell(w)}$ and
$\sgn^F(T_w) = \varepsilon_w = (-1)^{\ell(w)}$ for $w \in W$.


\begin{lemma}
\label{lemma:eigenvalues}
Let $F$ a subfield of $\complexes$. Let $\Gamma$ be
a $W$-digraph, and put $V = M(\Gamma)^F$.
Suppose $s \in S$ and
 $v = \sum_{\gamma\in X} \lambda_\gamma \gamma \in V$.
Then the following hold:
\begin{enumerate}[{\upshape(i)}]
\item
$T_s v = u^2 v$ if and only if
$\lambda_\beta = \lambda_\alpha$ whenever 
\solidedge{$\alpha$}{$\beta$}{$s$}
or
\dashededge{$\alpha$}{$\beta$}{$s$}
is an edge of $\Gamma$.
\item
$T_s v = - v$ if and only if
\[
\lambda_\beta
=
\begin{cases}
- u^{-2} \lambda_\alpha & 
\text{whenever $\solidedge{$\alpha$}{$\beta$}{$s$} \in \edges(\Gamma)$,} \\
- (u+1)(u^2-u)^{-1} \lambda_\alpha &
\text{whenever $\dashededge{$\alpha$}{$\beta$}{$s$} \in \edges(\Gamma)$.}
\end{cases}
\]
\end{enumerate}
\end{lemma}

\begin{proof}
The argument given for 
Lemma~2.4 in \cite{digraphpaper} applies, 
with $F$, $H^F$, and $V=M(\Gamma)^F$ replacing
$\rationals$, $H$, and $M=M(\Gamma)$, respectively.
\end{proof}


\begin{lemma}
\label{lemma:linearcharmults}
If $\Gamma$ is a $W$-digraph, $\vertices(\Gamma)$ is finite, 
and $F$ is a subfield of $\complexes$, 
then the following hold:
\begin{enumerate}[{\upshape(i)}]
\item
The number of connected components of $\Gamma$ is equal to 
$\charprod{M(\Gamma)^F}{\ind^F}{H^F}$.
\item
If $n(s,t) < \infinity$ for all $s, t \in S$, then
the number of acyclic connected components of $\Gamma$ is equal
to $\charprod{M(\Gamma)^F}{\sgn^F}{H^F}$.
\end{enumerate}
\end{lemma}

\begin{proof}
The proof for  Theorem~1.7 in \cite{digraphpaper} 
 applies here if
$\rationals$, $H$, and $M(\Gamma)$ are replaced by
by $F$, $H^F$, $M(\Gamma)^F$, respectively, using 
Lemma~\ref{lemma:eigenvalues} in place of  
\cite{digraphpaper}, Lemma~2.4.
\end{proof}


\begin{lemma}
If $F$ is a subfield of $\complexes$, $\Psi$ is a $W$-graph over $F$,  
$x \mapsto I_x \subseteq S$ is the vertex-labeling function for $\Psi$, 
 and $M(\Psi)^F_{\ind^F} \ne \set{0}$, 
then there is some $x_0 \in \vertices(\Psi)$ such that $I_{x_0}= \emptyset$.
\label{lemma:indlemma}
\end{lemma}

\begin{proof}
Let $X = \vertices(\Psi)$, and let 
$\mu$ be the edge-labeling function of $\Psi$. 
Suppose
$v = \sum_{x \in X} \gamma_x x \in M(\Psi)^F_{\ind^F}$ and $v \ne 0$.
Replacing $v$ by a scalar multiple if necessary, we can assume 
$\gamma_x \in F[u]$ for all $x\in X$ and
$\gcd \setof{\gamma_x}{x \in X} = 1$.
Choose $x_0 \in X$ such that $\gamma_{x_0} \not\in u F[u]$.
Suppose $I_{x_0} \ne \emptyset$, so that
 $s \in I_{x_0}$ for some $s \in S$.  Since 
\[
u^2 v = T_s v
 = 
- \sum_{  \genfrac {} {} {0pt} {3} {x \in X} {s \in I_x} } \gamma_x x
+
\sum_{  \genfrac {} {} {0pt} {3} {y \in X} {s \not\in I_y} }
 \gamma_y \left( u^2 y + u \sum_{\genfrac {} {} {0pt} {3} {z \in X} {s \in I_z}} \mu(z,x) z \right),
\]
comparing coefficients of $x_0$ shows $\gamma_{x_0} \in u F[u]$, and
so a contradiction is reached.  
Therefore $I_{x_0}= \emptyset$.
\end{proof}

\begin{proof}[Proof of Theorem~\ref{theorem:wgraphacyclic}]
For the remainder of this proof it is
assumed that  $n(s,t) < \infinity$ for $s, t \in S$, 
$\Gamma$ is a finite $W$-digraph, $\Psi$ is a $W$-graph
over the subfield $F$ of $\complexes$, 
and $M(\Gamma)^F$ is isomorphic to $M(\Psi)$ as $H^F$-modules.

Any connected component of $\Gamma$ that
contains a sink is acyclic by 
\cite{digraphpaper}, Theorem~1.5(ii).  On the other hand,
any acyclic connected component $C$ of 
$\Gamma$ contains some sink $\sigma$ because  
$\Gamma$ is finite, 
and $\sigma$ is the unique sink
in $C$ by \cite{digraphpaper}, Theorem~1.5(i).
Thus the number of sinks in $\Gamma$, that is, 
$N_\Gamma(S)$, is equal
to the number of acyclic connected components of $\Gamma$.
Thus by  
Lemma~\ref{lemma:linearcharmults}(ii),
$N_\Gamma(S)$ 
is equal to 
$\charprod{M(\Gamma)^F}{\sgn^F}{H^F}$. 
Also, 
$\charprod{M(\Psi)}{\sgn^F}{H^F} = N_{\Psi}(S)$
because
$M(\Psi)_{\sgn^F}$ has basis 
$\setof{x \in \vertices(\Psi)}{I_x = S}$ over $F(u)$. 
Hence 
\begin{equation*}
N_{\Gamma}(S) 
=
\charprod{M(\Gamma)^F}{\sgn^F}{H^F} 
 =
\charprod{M(\Psi)}{\sgn^F}{H^F}
= N_{\Psi}(S).
\end{equation*}

Now suppose $J \subseteq S$.  Let $\Psi_J$ be the $W_J$-graph
obtained from $\Psi$ by replacing
$I_x$ by $I_x \cap J$ for $x \in \vertices(\Psi)$.  
Also, let $\Gamma_J$ be the $W_J$-digraph obtained from 
$\Gamma$ by removing all edges with labels in
$S \setminus J$. 
Then
$M(\Gamma_J)^F \cong {M(\Gamma)^F \vert_{H_J^F}}
\cong {M(\Psi) \vert_{H_J^F}} \cong M(\Psi_J)$ as $H_J^F$-modules, 
and so the reasoning above gives
$N_{\Gamma_J}(J) = N_{\Psi_J}(J)$.
Therefore 
\[
\sum_{J \subseteq K \subseteq S} N_{\Gamma}(K)
=
N_{\Gamma_J}(J)
=
N_{\Psi_J}(J)
=
\sum_{J \subseteq K \subseteq S} N_{\Psi}(K).
\]
(Note $S$ itself is finite because $\edges(\Gamma)$ is
finite, so the sums above are finite.)
Thus 
part (i) of the theorem holds
by induction on $\abs{S \setminus J}$.

Let $M_0$ be the $H^F$-submodule of $M(\Psi)$ 
with basis
$\setof{x \in \vertices(\Psi)}{I_x \ne \emptyset}$.
By Lemma~\ref{lemma:indlemma},  
\[
M(\Psi)_{\ind^F} \cap M_0 
= 
(M_0)_{\ind^F}
=
\set{0}.
\]
Thus
\begin{equation*}
\begin{split}
\charprod{M(\Gamma)^F}{\ind^F}{H^F} 
 & =
\charprod{M(\Psi)}{\ind^F}{H^F} 
 =
\dim_{F(u)} M(\Psi)_{\ind^F} \\
& \le 
\dim_{F(u)} \left( M(\Psi) / M_0 \right) 
 =
\abs{ \setof{ x \in \vertices(\Psi) } {I_x = \emptyset }} \\
& =
N_{\Psi}(\emptyset)
=
N_{\Gamma}(\emptyset),
\end{split}
\end{equation*}
with the last equality holding by part (i) of the theorem.  
Now, 
$\charprod{M(\Gamma)^F}{\ind^F}{H^F}$ is equal to the number of
connected components of $\Gamma$ by
Lemma~\ref{lemma:linearcharmults}(i), while
 $N_\Gamma(\emptyset)$ is equal to 
the number of sources of $\Gamma$. 
Therefore
$\Gamma$ has at least as many
sources as connected components. 
Because each connected component contains
at most one source by \cite{digraphpaper}, Theorem~1.5(i), it follows that 
every connected component of $\Gamma$  contains a 
(unique) source.  Hence every connected component of $\Gamma$
 is acyclic by \cite{digraphpaper}, Theorem~1.5(ii).
Therefore $\Gamma$ itself is acyclic, so part (ii) of the theorem holds and the 
proof of the theorem is complete.
\end{proof}


\section{Errata from \cite{digraphpaper}}
None of the errata listed here has an effect on the results of  \cite{digraphpaper}. 

\medskip\bigskip\noindent Page 315, line 4:  The index of summation in the last sum should be $i$, 
not $\ell$. 
The displayed formula containing this line should read as follows: 
\begin{equation*}
\begin{split}
T_{s_\ell^{-1}} {\widetilde{\eta}}_\ell
& = 
T_{s_\ell^{-1}} ({\widetilde{\varphi}}_\ell + u {\widetilde{\eta}}_{\ell-1}) 
= 
T_{s_\ell^{-1}} {\widetilde{\varphi}}_\ell + u T_s T_{t_{\ell-1}^{-1}} 
     {\widetilde{\eta}}_{\ell-1} \\
& = 
u^{2\ell} T_e + T_{s_{2\ell}^{-1}}
 + u T_s \sum_{i=0}^{2\ell-2} u^{i} T_{t_{2\ell-i-2}^{-1}} \\
& = 
 u^{2\ell} T_e + T_{s_{2\ell}^{-1}} 
   + \sum_{i=0}^{2\ell-2} u^{i+1} T_{s_{2\ell-(i+1)}^{-1}} 
= 
\sum_{i=0}^{2\ell} u^{i} T_{s_{2\ell-i}^{-1}}
\end{split}
\end{equation*}

\medskip\bigskip\noindent Page 316, line 3:  
Replace $\mu_m$ by $\mu_{m}^\prime$.   
The displayed equations containing this line  
should read as follows:
\begin{equation*}
\begin{cases}
\mu_1  = T_s \mu_0, \mu_2 = T_t \mu_1, \cdots, 
  \mu_{m-1} = T_{s^\prime} \mu_{m-2}, \mu_m = T_{t^\prime} \mu_{m-1}, & \\
\mu_1^\prime = T_t \mu_0, \mu_2^\prime = T_s \mu_1^\prime, \cdots, 
  \mu_{m-1}^\prime = T_{t^\prime} \mu_{m-2}^\prime, 
  \mu_{m}^\prime = T_{s^\prime} \mu_{m-1}^\prime. &
\end{cases}
\end{equation*}

\medskip\bigskip\noindent Page 321, lines 7 and 9 from bottom:  
The edges in each of these lines should be reversed.
The displayed formula containing these lines should read as follows:
\[
\zeta_i
=
\begin{cases}
-\dfrac{1}{u^2}\zeta_{i-1} & 
\text{if \solidedge{$\gamma_{i-1}$}{$\gamma_{i}$}{$s$} 
or \solidedge{$\gamma_{i-1}$}{$\gamma_{i}$}{$t$} is an edge of $\Gamma$,} \\
-u^{2}\zeta_{i-1} & 
\text{if \reversedsolidedge{$\gamma_{i-1}$}{$\gamma_{i}$}{$s$}
or \reversedsolidedge{$\gamma_{i-1}$}{$\gamma_{i}$}{$t$} is an edge of $\Gamma$,} \\
- \dfrac{u+1}{u^2-u}  \zeta_{i-1} &
\text{if \dashededge{$\gamma_{i-1}$}{$\gamma_{i}$}{$s$}
or \dashededge{$\gamma_{i-1}$}{$\gamma_{i}$}{$t$} is an edge of $\Gamma$,} \\
- \dfrac{u^2-u}{u+1}  \zeta_{i-1} &
\text{if \reverseddashededge{$\gamma_{i-1}$}{$\gamma_{i}$}{$s$}
or \reverseddashededge{$\gamma_{i-1}$}{$\gamma_{i}$}{$t$} is an edge of $\Gamma$.} 
\end{cases}
\]

\medskip\bigskip\noindent Page 322, lines 2--4:  The direction of second and fourth edges should be reversed.  The sentence containing these lines should read as follows:

It follows that the number of edges of type
\solidedge{$\gamma_{i-1}$}{$\gamma_{i}$}{} 
(labeled either $s$ or $t$)
 is equal to the number of edges of type
\reversedsolidedge{$\gamma_{i-1}$}{$\gamma_{i}$}{}, $1 \le i \le 2m$, 
and the number of edges of type
\dashededge{$\gamma_{i-1}$}{$\gamma_{i}$}{}
is equal to the number of edges of type
\reverseddashededge{$\gamma_{i-1}$}{$\gamma_{i}$}{}, $1 \le i \le 2m$.  

\medskip\bigskip\noindent Page 323, lines 4--6:  The direction of second and fourth edges should be reversed, and the label should be removed from the 
first edge.  The sentence containing these lines should read as follows:

Since $\Gamma$ has a unique sink $\beta$ 
and the number of edges of type 
\solidedge{$\gamma_{i-1}$}{$\gamma_{i}$}{} 
 is equal to the number of edges of type
\reversedsolidedge{$\gamma_{i-1}$}{$\gamma_{i}$}{}, $1 \le i \le 2m$, 
and the number of edges of type
\dashededge{$\gamma_{i-1}$}{$\gamma_{i}$}{}
is equal to the number of edges of type
\reverseddashededge{$\gamma_{i-1}$}{$\gamma_{i}$}{}, $1 \le i \le 2m$, 
it follows that
$\beta = \gamma_m$ is opposite to $\alpha$. 

\medskip\bigskip\noindent Page 327, line 3 from bottom:  
Replace $\Gamma_s$ by $\Gamma_{\set{s}}$.

\medskip\bigskip\noindent Page 332, line 8:  Replace $\solidedge{$\alpha$}{$\beta$}{$s$}$
by $\reversedsolidedge{$\alpha$}{$\beta$}{$s$}$.  The relevant sentence reads as follows:

If  $\reversedsolidedge{$\alpha$}{$\beta$}{$s$} \in \edges(\Gammaarrow)$ for some
$s \in S$, then 
$\alpha \in \closedray{\sigma}$ because
$\beta \in \closedray{\sigma}$ and $\alpha \in \closedray{\beta}$.


\bibliographystyle{plain}
\bibliography{combined}


\vfil\eject
\enddocument
\bye
\bye